\documentclass[11pt,psamsfonts]{amsart}
\usepackage{amsmath}
\usepackage{amsthm}
\usepackage{amssymb}
\usepackage{amscd}
\usepackage{amsfonts}
\usepackage{amsbsy}
\usepackage{graphicx}
\usepackage[dvips]{psfrag}
\usepackage{array}
\usepackage{color}
\usepackage{epsfig}
\usepackage{url}
\usepackage{overpic}
\usepackage{epstopdf}
\usepackage{float}
\usepackage{enumerate}
\usepackage{multirow}
\usepackage{rotating}
\usepackage{hyperref}

\newcommand{\R}{\ensuremath{\mathbb{R}}}

\newcommand{\F}{\ensuremath{\mathcal{F}}}

\usepackage{mathtools}
\usepackage{subfigure} 
\usepackage{pgfplots}

\def\e{\varepsilon}

\newtheorem {theorem} {Theorem}

\newtheorem {proposition} [theorem]{Proposition}

\newtheorem {lemma}  [theorem]{Lemma}

\usepackage{mathpazo}

\begin{document}
\renewcommand{\arraystretch}{1.5}

\title[Perturbing periodic integral manifold of non-smooth differential systems]
{Limit cycles bifurcating from periodic integral manifold\\ in non-smooth differential systems}

\author[O.A.R. Cespedes]{Oscar A. R. Cespedes$^1$}
\address{$^1$ Departamento de Matem\'{a}tica, Facultad de Ciencias, Universidad Antonio Nariño, Bogotá, Colombia} 
\email{oramirez179@uan.edu.co} 

\author[D.D. Novaes]{Douglas D. Novaes$^2$}
\address{$^2$ Departamento de Matem\'{a}tica, Instituto de Matem\'{a}tica, Estat\'{i}stica e Computa\c{c}\~{a}o Cient\'{i}fica, Universidade
Estadual de Campinas, \ Rua S\'{e}rgio Buarque de Holanda, 651, Cidade Universit\'{a}ria Zeferino Vaz, 13083-859, Campinas, SP,
Brazil} 
\email{ddnovaes@unicamp.br}



\maketitle

\begin{abstract}	
This paper addresses the perturbation of higher-dimensional non-smooth autonomous differential systems characterized by two zones separated by a codimension-one manifold, with an integral manifold foliated by crossing periodic solutions. Our primary focus is on developing the Melnikov method to analyze the emergence of limit cycles originating from the periodic integral manifold. While previous studies have explored the Melnikov method for autonomous perturbations of non-smooth differential systems with a linear switching manifold and with a periodic integral manifold, either open or of codimension 1, our work extends to non-smooth differential systems with a non-linear switching manifold and more general periodic integral manifolds, where the persistence of periodic orbits is of interest. We illustrate our findings through several examples, highlighting the applicability and significance of our main result.
\end{abstract}

\section{Introduction and  statements of the main results}
One of the principal challenges in the qualitative theory of differential systems in higher-dimensions is the investigation of the existence of invariant sets, particularly periodic orbits. This paper focuses on the study of periodic solutions of $(m+1)$-dimensional (with $m>1$) piecewise smooth autonomous differential systems of the form
  \begin{equation}\label{s0} 
        \left(\dot{x}, \dot{y}, \dot{z}\right)^T=X_0(x,y,z)+\varepsilon X_1(x,y,z)+\varepsilon^{2}R(x,y,z,\varepsilon), \quad (x,y,z)\in D,
    \end{equation}
where, for $m>n$, $D=D_1\times D_2 \times D_3\subset \R^{n}\times \R^{m-n}\times \R \cong \R^{m+1}$ is an open set,
$X_i(x,y,z)$ (for $i=0,1$) and $R(x,y,z,\varepsilon)$ are defined as the following piecewise functions
    \[
       X_i(x,y,z)=\left\{
            \begin{array}{cl}
               X_{i}^+(x,y,z), & z>g(x,y), \\
                X_{i}^-(x,y,z), & z<g(x,y), \\
            \end{array}
        \right.
    \]
   for $i=0,1$, and
    \[
        R(x,y,z,\varepsilon)=\left\{
            \begin{array}{cl}
                R^+(x,y,z,\varepsilon), & z>g(x,y), \\
                R^-(x,y,z,\varepsilon), & z<g(x,y),
            \end{array}
        \right.
    \]
with $X^{\pm}_{i}:D\rightarrow \R^n\times \R^{m-n}\times \R \cong \R^{m+1}$, $R^{\pm}:D\times (-\e_0,\e_0)\rightarrow\R^n\times \R^{m-n}\times \R \cong \R^{m+1}$, and $g:D_1\times D_2\rightarrow D_3$ being $\mathcal{C}^1$ functions. In this case, the switching manifold is given by $$\Sigma=\{(x,y,z)\in D:\, z=g(x,y)\}.$$ 

Our research is grounded in the fundamental hypothesis that the unperturbed vector field $X_0$ contains an $n$-dimensional manifold of initial conditions $\mathcal{Z} \subset \Sigma$, where the orbits are periodic and intersect $\Sigma$ transversely. The saturation of $\mathcal{Z}$ through the flow of $X_0$ forms a periodic integral manifold of $X_0$, which we will denote by $\widetilde{\mathcal{Z}}$. Our primary objective is to investigate the emergence of limit cycles in the perturbed system \eqref{s0}, originating from this periodic integral manifold $\widetilde{\mathcal{Z}}$. In the smooth context, this kind of problem has been addressed by Malkin \cite{Ma}, Rosseau \cite{Ro} and, more recently, in \cite{BFL,BGL,BGL2,CLN,JN3,RC}.

The Averaging Theory, a classical method for studying such problems, has been recently extended to piecewise smooth nonautonomous differential systems (see, for instance, \cite{BBLN19, LMN15, LliNovRod2017,LliNovRod2020,LNT2014}. However, applying the averaging method to autonomous differential systems like \eqref{s0} requires a suitable change of variables to transform the differential system into a non-autonomous periodic form. This transformation can be challenging, especially when the unperturbed differential system $X_0$ does not exhibit a linear center.

Another valuable tool for addressing this problem is the Melnikov method, which, under appropriate conditions, can be applied directly to the differential system \eqref{s0}. Previous studies have explored the Melnikov method for autonomous perturbations of non-smooth differential systems with a linear switching manifold and with a periodic integral manifold, either open or of codimension 1. For instance, in \cite{GLN}, the first-order Melnikov function was obtained for differential system \eqref{s0} by assuming that the switching manifold $\Sigma$ and $\widetilde{\mathcal{Z}}$ are hyperplanes. Assuming that $\Sigma$ is a hyperplane and $\widetilde{\mathcal{Z}}$ is an open submanifold of $D$, the first-order Melnikov method was developed in \cite{THM} for near-integrable differential systems, while in \cite{CLL}, the Melnikov method was developed to any order.

In this paper, we focus in extending the Melnikov method to non-smooth differential systems \eqref{s0} with non-linear switching manifolds and more general periodic integral manifolds, where the persistence of periodic orbits is of interest. Our main result is stated in Section \ref{sec:mf}. In Sections \ref{sec:ex1} and \ref{sec:ex2}, we shall illustrate our findings through several examples, highlighting the applicability and significance of our main result.

\subsection{The Melnikov function}\label{sec:mf}
 In order to state our main results, we need to introduce some preliminary concepts, assumptions, and notations.
Let $\mathcal{V}$ an open bounded subset of $\mathcal{\R}^n$ and $v:\overline{\mathcal{V}} \rightarrow \mathcal{\R}^{m-n}$ a $\mathcal{C}^{2}$ function such that
$$\mathcal{Z}=\text{Graph}(g|_{\mathcal{U}}),$$
where
 $$\mathcal{U}=\{(u, v(u)): u \in \overline{\mathcal{V}}\}\subset D_1\times D_2. $$
Denote by $\varphi_0(t, x,y,z)$ the solution of $(\dot{x}, \dot{y},\dot{z})^T=X_0(x,y,z)$, such that $\varphi_0(0, x,y,z)=(x,y,z)$. Additionally, let $\varphi_0^{\pm}(t, x,y,z)$ the solution of $(\dot{x}, \dot{y},\dot{z})^T=X^{\pm}_0(x,y,z)$, with initial condition $(x,y,z)$ at $t=0$. We assume the
following hypothesis:
 \begin{description}
    \item[(H1)]For each $u\in \mathcal{V}$, there exist $\tau^-(u)<0$ and $\tau^+(u)>0$ such that  
    \begin{itemize}
    \item[(a)] $\varphi_0^-(\tau^-(u),u, v(u), g(u,v(u)) )=\varphi_0^+(\tau^+(u),u, v(u), g(u,v(u)) ),$
\item[(b)] $h(\varphi_0^-(t,u, v(u), g(u,v(u)) ))<0,$ for $t \in( \tau^-(u),0)$,
\item[(c)] $h(\varphi_0^+(t,u, v(u), g(u,v(u)) ))>0,$  $t \in(0, \tau^+(u))$,
\item[(d)] $(X_0^{\pm} h)(\varphi^{+}_0(t, u, v(u), g(u, v(u))))>0$ for $t=0$,
\item[(e)] $(X_0^{\pm} h)(\varphi^{-}_0(t, u, v(u), g(u, v(u))))<0,$ for  $t=\tau^{-}(u)$,
\end{itemize}
where $h(x,y,z)=z-g(x,y)$ and $ (X_0^\pm h)(p)$
 denotes the Lie derivative of $h$ at $p$ in the direction of the vector field $X_0^{\pm}.$
\end{description}

Under hypothesis \textbf{(H1)}, it follows that for each $u\in \mathcal{V}$, the differential system \eqref{s0} for $\varepsilon=0$ admits a crossing periodic orbit $L_u = L_u^+ \cup L_u^-$, composed of an orbit segment of $X^+_0$, denoted as $L_u^+$, and an orbit segment of $X^-_0$, denoted as $L_u^-$, where,
\begin{equation*}
\begin{aligned}
 L_u^+ &= \{\varphi_0^-(t, u, v(u), g(u,v(u))): \tau^-(u) \leq t \leq 0\}, \\
 L_u^- &=  \{\varphi_0^+(t, u, v(u), g(u,v(u))): 0 \leq t \leq \tau^+(u)\}.
\end{aligned}
\end{equation*}
Moreover, $L_u^+$ and  $L_u^-$ intersect $\Sigma$ precisely at two points: $(u, v(u), g(u,v(u)))$ and $$\varphi_0^-(\tau^-(u), u, v(u), g(u,v(u)))=\varphi_0^+(\tau^+(u), u, v(u), g(u,v(u))),$$ and these intersections are transversal.

Now, consider the projection functions $\pi_1: \mathbb{R}^{n}\times \mathbb{R}^{m-n}\times \mathbb{R}\rightarrow \mathbb{R}^n$ and   $\pi_2: \mathbb{R}^{n}\times \mathbb{R}^{m-n}\times \mathbb{R}\rightarrow \mathbb{R}^{m-n}$  given by
$\pi_1(x,y,z)=x$ and $\pi_2(x,y,z)=y$. Given $\ell \in \mathbb{Z}^+$, we denote by $\text{Id}_{\ell}$ the $\ell\times\ell$ identity matrix and denote by $\Pi_{i}$ the matrix associated with $\pi_i$.

Finally, define the bifurcation function $\mathcal{M}: \mathcal{V}\rightarrow \mathbb{R}^n$ as follows:
\begin{equation}\label{Mel}
\begin{aligned}
\mathcal{M}(u) &=  \Pi_1 \cdot \left(\text{Id}_{m+1}-\boldsymbol{\beta}(u)  \cdot \left(\Pi_2 \cdot \boldsymbol{\beta}(u) \right)^{-1} \cdot \Pi_2\right)\cdot \boldsymbol{\alpha}(u),
\end{aligned}
\end{equation}
with 
$$\boldsymbol{\alpha}(u)= \boldsymbol{\alpha}^+(u) - \boldsymbol{\alpha}^-(u)\quad \text{and}\quad \boldsymbol{\beta}(u)= \boldsymbol{\beta}^+(u) - \boldsymbol{\beta}^-(u),$$
where
\begin{equation}\label{beta}
\begin{aligned}
\boldsymbol{\alpha}^\pm(u)&= \left( \text{Id}_{m+1} - \frac{X_0^\pm(\varphi_0^\pm \circ P^\pm(u)) \cdot \nabla h(\varphi_0^\pm \circ P^\pm(u))}{X_0^\pm h(\varphi_0^\pm \circ P^\pm(u))} \right) \cdot \omega^\pm \circ P^\pm(u),\\
\boldsymbol{\beta}^\pm(u) &= \left( \text{Id}_{m+1} - \frac{X_0^\pm(\varphi_0^\pm \circ P^\pm(u)) \cdot \nabla h(\varphi_0^\pm \circ P^\pm(u))}{X_0^\pm h(\varphi_0^\pm \circ P^\pm(u))} \right) \cdot Y^\pm \circ P^\pm(u),
\end{aligned}
\end{equation}
with $P^{\pm}(u) = (\tau^{\pm}(u),u,v(u),g(u,v(u)))$,
\[
Y^\pm(t,x,y,z) = \frac{\partial \varphi_0^\pm}{\partial y}(t,x,y,z) + \frac{\partial \varphi_0^\pm}{\partial z}(t,x,y,z) \cdot \frac{\partial g}{\partial y}(x,y),
\]
and
$$\omega^{\pm}(t,x,y,z)= D_{(x,y,z)} \varphi^{\pm}_0(t,x,y,z) \int_{0}^t (D_{(x,y,z)} \varphi^{\pm}_0(s,x,y,z))^{-1}\cdot X_1^{\pm}( \varphi^{\pm}_0(s,x,y,z))ds.$$

Now, we are ready to state our main result about the persistence of periodic orbits.
\begin{theorem}\label{thm}
In addition to hypothesis {\bf (H1)}, we assume that $\text{det}(\Pi_2 \cdot \boldsymbol{\beta}(u)) \neq 0$ for every $u\in \overline{\mathcal{V}}$. Then, for
each $u^*\in \mathcal{V}$, such that $\mathcal{M}(u^{*}) = 0$ and $\text{det}(D_u\mathcal{M})(u^{ *})\neq 0$, and $|\e|$ sufficiently small there is a unique crossing periodic solution $\varphi_{\varepsilon}(t)$ of \eqref{s0} satisfying $|\varphi_{\varepsilon}(0)- (u^*, v (u^*), g(u^*, v(u^*)))| \rightarrow 0$ when $\varepsilon \rightarrow 0$.
\end{theorem}

It is important to mention that Theorem \eqref{thm} generalizes the main result of \cite{GLN}.

\subsection{ Perturbations of 3D piecewise linear systems with a isochronous plane}\label{sec:ex1}
Consider the following 3D discontinuous piecewise linear differential system
\begin{equation}\label{example1}
\left(\dot{x}, \dot{y}, \dot{z}\right)^T=\left\{
\begin{array}{ll}
X_{i,0}^+(x,y,z)+\varepsilon X_1^+(x,y,z),& f(x,y,z)>0,\\
X_{i,0}^-(x,y,z)+\varepsilon X_1^-(x,y,z),& f(x,y,z)<0,
\end{array}\right.
\end{equation}
where  $i\in\{a,b\}$, $f: \R^3\rightarrow \R$ such that $\Sigma=f^{-1}(0)$ is a submanifold of codimension 1 of $\R^3$,
\begin{equation*}
X_{a,0}^{\pm}(x,y,z)=
\left(
\begin{array}{c}
x\\
-z\mp 1\\
x+y
\end{array}
\right), \quad
X_{b,0}^{\pm}(x,y,z)= \left(
\begin{array}{c}
x\\
z\mp 1\\
x+y
\end{array}\right), 
\end{equation*}
and
\begin{equation*}
X_1^{\pm}(x,y,z)=
\left(
\begin{array}{l}
\alpha_0^\pm + \alpha_1^\pm x + \alpha_2^\pm y + \alpha_3^\pm z \\
\beta_0^\pm + \beta_1^\pm x + \beta_2^\pm y + \beta_3^\pm z \\
\kappa_0^\pm + \kappa_1^\pm x + \kappa_2^\pm y + \kappa_3^\pm z
\end{array}\right)
\end{equation*}

In \cite{GLN} Llibre, Novaes and Gouveia considered that $f(x,y,z)=z$ and they proved that unperturbed differential system \eqref{example1}, when $\varepsilon=0$, has an invariant plane at $\mathcal{S}=\{(x,y,z)\in \R^3: x=0\}$ containing a period annulus $\mathcal{A}$ foliated by crossing periodic orbits, see Figure \ref{Fig1}. Additionally, they studied the number of the limit cycle that bifurcates from the periodic orbits of the invariant plane $\mathcal{S}$.  However, Propositions 7 and 8 in \cite{GLN}  present  some inconsistencies concerning the number of bifurcating limit cycles.  In the following result, as an application of Theorem \ref{thm}, we rectify these inconsistencies.

\begin{figure}[H]
	\begin{center}
\subfigure[$i=a$]{
\begin{overpic}[width=5cm]{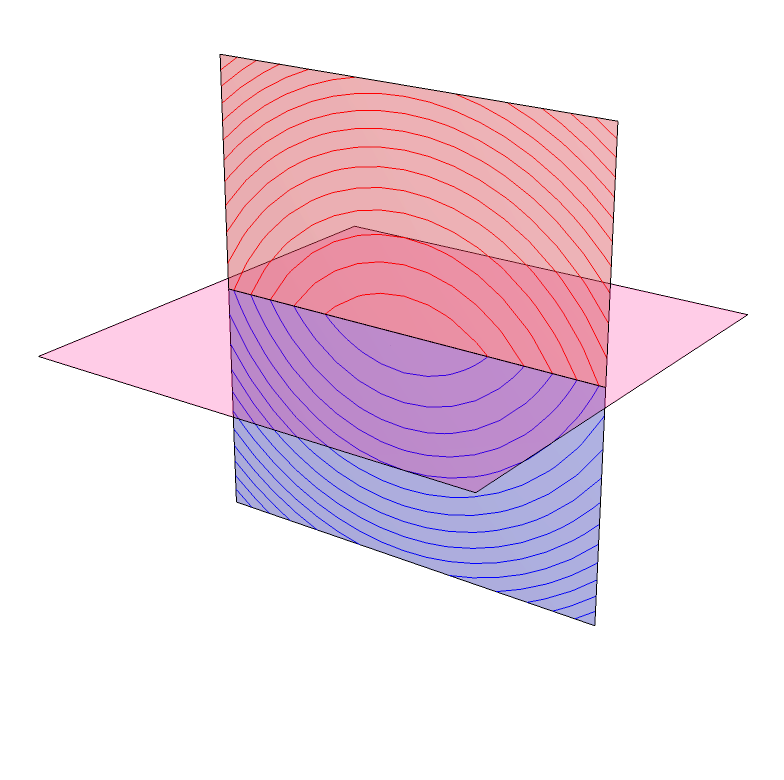}
 \put(85,45){$\Sigma$}
  \put(80,20){$\mathcal{S}$}
\end{overpic}}\qquad
\subfigure[$i=b$]{
\begin{overpic}[width=5cm]{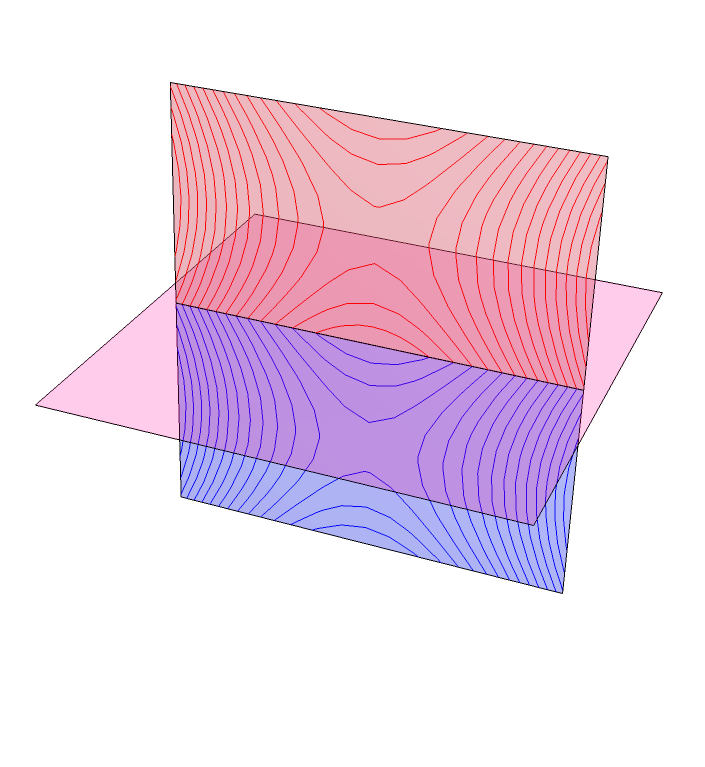}
 \put(85,45){$\Sigma$}
  \put(80,20){$\mathcal{S}$}
\end{overpic}}
\caption{Phase Portrait of unperturbed differential system \eqref{example1} on the invariant plane $\mathcal{S}$, for $i\in\{a,b\}$ and $f(x,y,z)= z$.}\label{Fig1}
\end{center}
\end{figure}

\begin{theorem}\label{p1}
For $f(x,y,z)=z$,  $i \in\{a,b\}$ and  $\varepsilon\neq 0$ sufficiently small,  differential system \eqref{example1} admits at least 2  limit cycle converging, when $\varepsilon\rightarrow0,$ to some of the periodic orbits contained in the plane $\mathcal{S}=\{(x,y,z)\in \R^3: x=0\}$.
\end{theorem}

We now investigate the impact of the nonlinearity of the switching manifold $\Sigma$ on the number of periodic orbits that can bifurcate from the invariant plane $\mathcal{S}$. To do so, we consider the function $f(x,y,z)=-d x^2 - c x y - y^2 + z$, where $c,d\in\{\pm1,0\}$, as illustrated in Figure \ref{Fig2}. In the forthcoming result, we characterize the submanifold of $\mathcal{S}$ that is foliated by crossing periodic orbits when $\varepsilon=0$ and $f(x,y,z)=-d x^2 - c x y - y^2 + z$. Additionally, we study the bifurcation of limit cycles by applying Theorem \ref{thm}.

\begin{figure}[H]
	\begin{center}
\subfigure[$(c,d)=(0,0)$]{
\begin{overpic}[width=4cm]{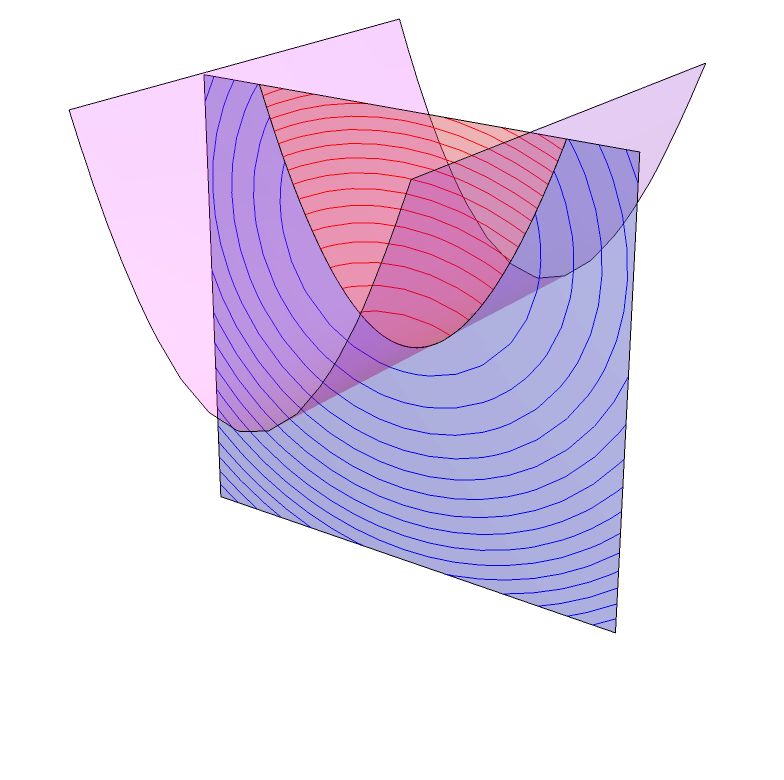}
  \put(90,80){$\Sigma$}
  \put(80,20){$\mathcal{S}$}
\end{overpic}}\qquad
\subfigure[$(c,d)=(0,-1)$]{
\begin{overpic}[width=4cm]{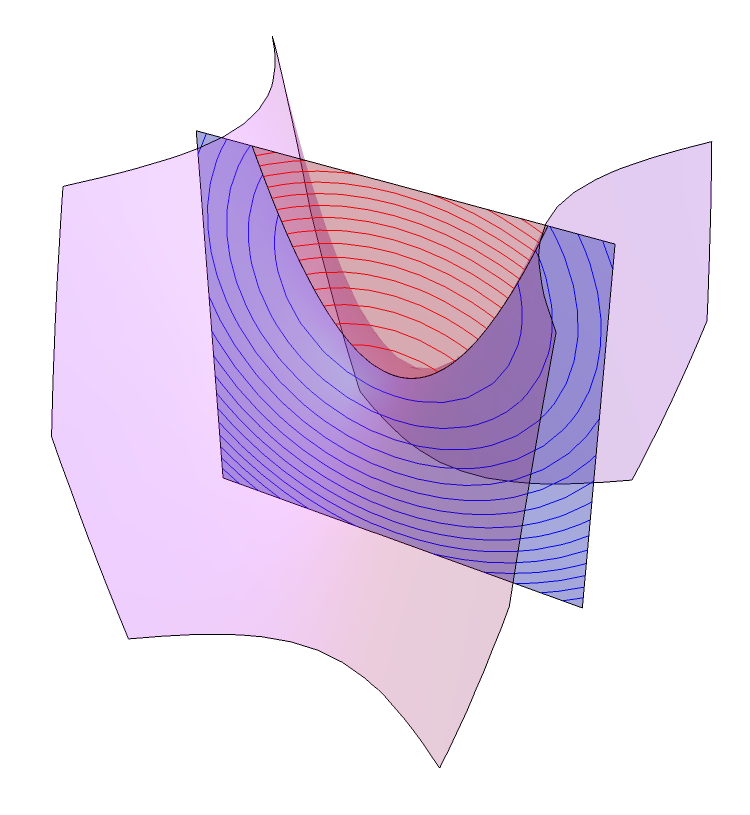}
 \put(90,80){$\Sigma$}
  \put(75,20){$\mathcal{S}$}
\end{overpic}}
\subfigure[$(c,d)=(0,1)$]{
\begin{overpic}[width=4cm]{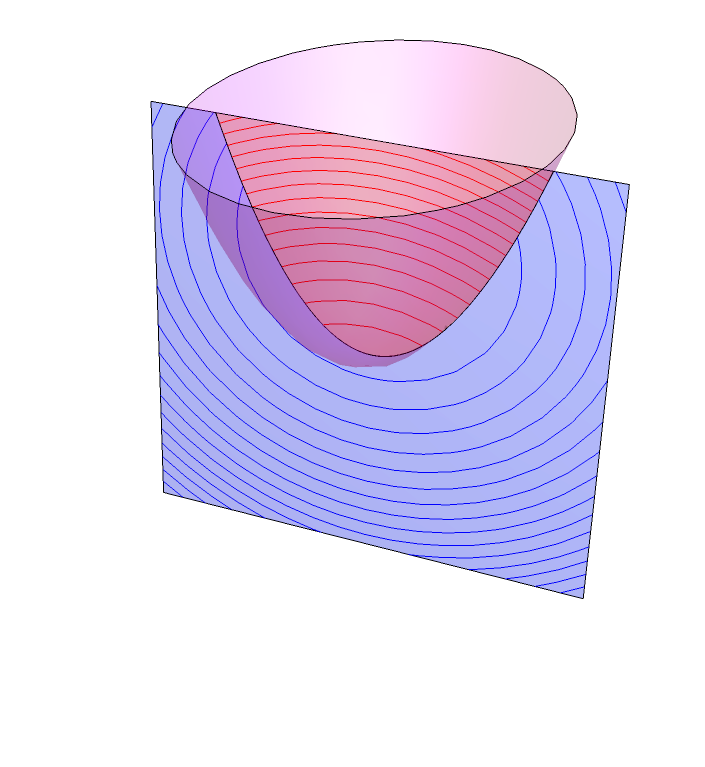}
 \put(76,80){$\Sigma$}
  \put(80,20){$\mathcal{S}$}
\end{overpic}}
\caption{ Phase Portrait of unperturbed differential system \eqref{example1} on the invariant plane $\mathcal{S}$, for $i=a$ and $f(x,y,z)=-d x^2 - c x y - y^2 + z$.}\label{Fig2}
\end{center}
\end{figure}	

\begin{theorem} \label{quadratic}
For $f(x,y,z)=-d x^2 - c x y - y^2 + z$,  $i=a$ and $\varepsilon=0$, unperturbed differential system \eqref{example1}  has an invariant plane at $\mathcal{S}=\{(x,y,z)\in \R^3: x=0\}$ containing a period annulus $\mathcal{A}=\mathcal{A}_1 \cup \mathcal{A}_2$  foliated by crossing periodic orbits, see Fig \ref{Fig3},  where
$$\mathcal{A}_j=
\left\{
\begin{array}{ll}
\{(0,y, z) \in \R^3: y^2+(z-1)^2>1, z\leq y^2\},& j=1,\\
\{(0,y, z) \in \R^3
 :  y^2+(z+1)^2>5, z\geq y^2\},& j=2.
\end{array}\right.$$
Furthermore, for $\varepsilon\neq 0$ sufficiently small, the differential system \eqref{example1} admits at least 7  limit cycles converging, when $\varepsilon\rightarrow0$, to some of the periodic orbits contained in the plane $\mathcal{S}=\{(x,y,z)\in \R^3: x=0\}.$
\end{theorem}
\begin{figure}[H]
\begin{overpic}[width=1.5in]{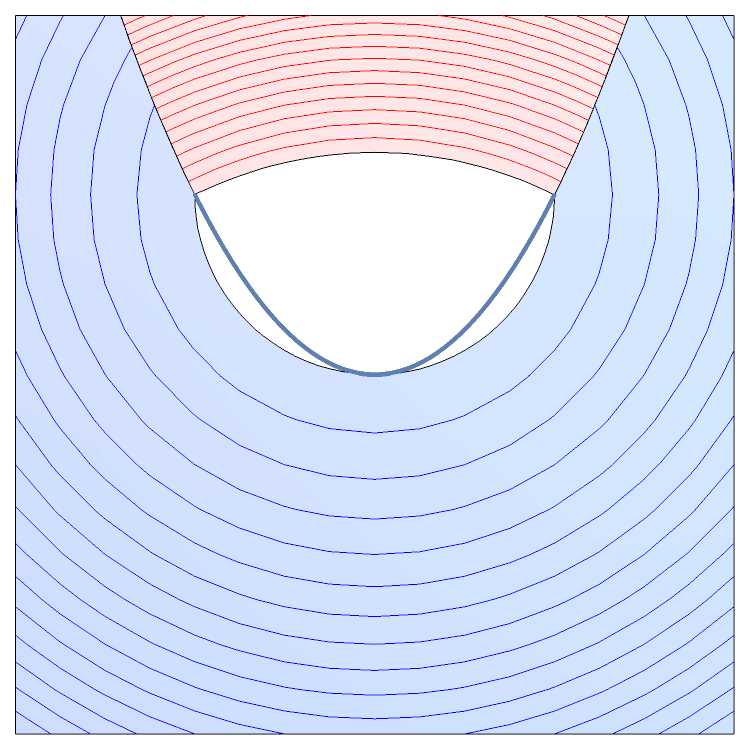}
 \put(78,80){$\Sigma$}
  \put(80,20){$\mathcal{S}$}
\end{overpic}
\caption{Annulus region on invariant plane of unperturbed differential system \eqref{example1} for $i=a$ and $f(x,y,z)=-d x^2 - c x y - y^2 + z.$}\label{Fig3}
\end{figure}	

\subsection{Perturbations of 3D piecewise quadratic systems with an isochronous paraboloid}\label{sec:ex2}
Consider the $3D$  piecewise quadratic polynomial differential system
\begin{equation}\label{example2}
\left(
\begin{array}{l}
\dot{x}\\
\dot{y}\\
\dot{z}\\
\end{array}\right)=
\left(
\begin{array}{c}
-y\\
x\\
\lambda(x^2+y^2-z)
\end{array}
\right)+\varepsilon
\left(
\begin{array}{l}
P_1(x,y,z)\\
 P_2(x,y,z)\\
P_3(x,y,z)
\end{array}\right),
\end{equation}
where
\begin{equation*}
P_{\ell}(x,y,z)=\left\{
\begin{array}{ll}
P_\ell^+(x,y,z)=\displaystyle \sum_{0\leq i+j+k\leq 2} p^{\ell,+}_{ijk} x^i y^jz^k, & f(x,y,z)>0,\\
\\
P_\ell^-(x,y,z)=\displaystyle \sum_{0\leq i+j+k\leq 2} p^{\ell,-}_{ijk} x^i y^jz^k,& f(x,y,z)<0,\\
\end{array}\right.
\end{equation*}
with $\lambda, p^{\pm}_{ijk}\in \R$, $\lambda\neq0$ and $f:\R^3\rightarrow \R$ such that $\Sigma=f^{-1}(0)$ is a manifold of codimension one of $\R^3$. 

It is notable that the unperturbed differential system \eqref{example2} possesses an isochronous invariant paraboloid $\mathcal{S}=\{(x,y,z): z=x^2+y^2\}$, meaning that $\mathcal{S}$ is foliated by periodic orbits with identical periods. Utilizing averaging theory, Llibre, Rebollo, and Torregrosa in \cite{LRT} examined the smooth case, wherein $P_i^-=P_i^+$ for $i\in{1,2,3}$, and determined that the differential system has at least $2$ limit cycles, which bifurcate from the periodic orbits of the invariant isochronous surface $\mathcal{S}$. Subsequently, in the nonsmooth case, we investigate the number of limit cycles that bifurcate from crossing periodic orbits in $\mathcal{S}$ if $f(x,y,z) \in \{y,y-z, y-x^2\}$. In the ensuing result, the proof of which is straightforward and will be omitted, we characterize the submanifold of $S$ that is foliated by crossing periodic orbits.

\begin{figure}[H]
	\begin{center}
\subfigure[$f(x,y,z)=y$]{
\begin{overpic}[width=1.5in]{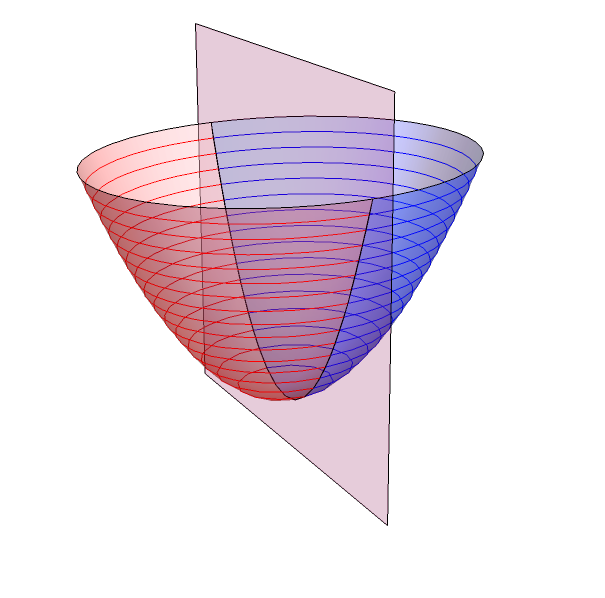}
 \put(65,15){$\Sigma$}
  \put(15,80){$\mathcal{S}$}
\end{overpic}}\quad
\subfigure[$f(x,y,z)=y-z$]{
\begin{overpic}[width=1.5in]{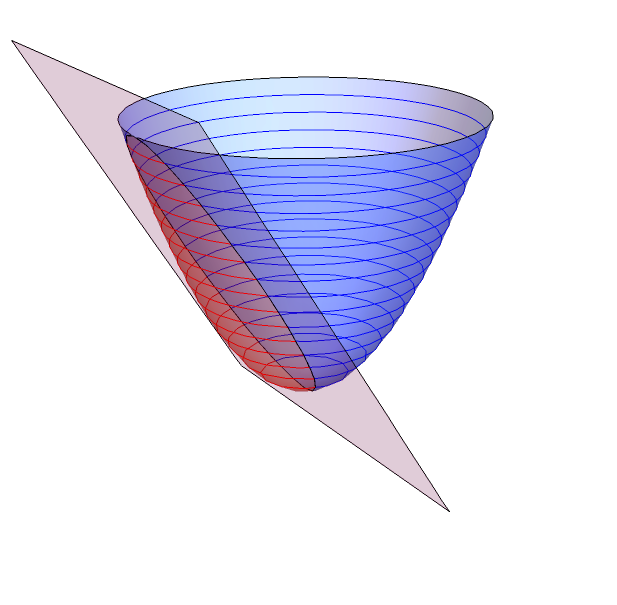}
\put(75,15){$\Sigma$}
  \put(35,87){$\mathcal{S}$}
\end{overpic}}\quad
\subfigure[$f(x,y,z)=y-x^2$]{
\begin{overpic}[width=1.5in]{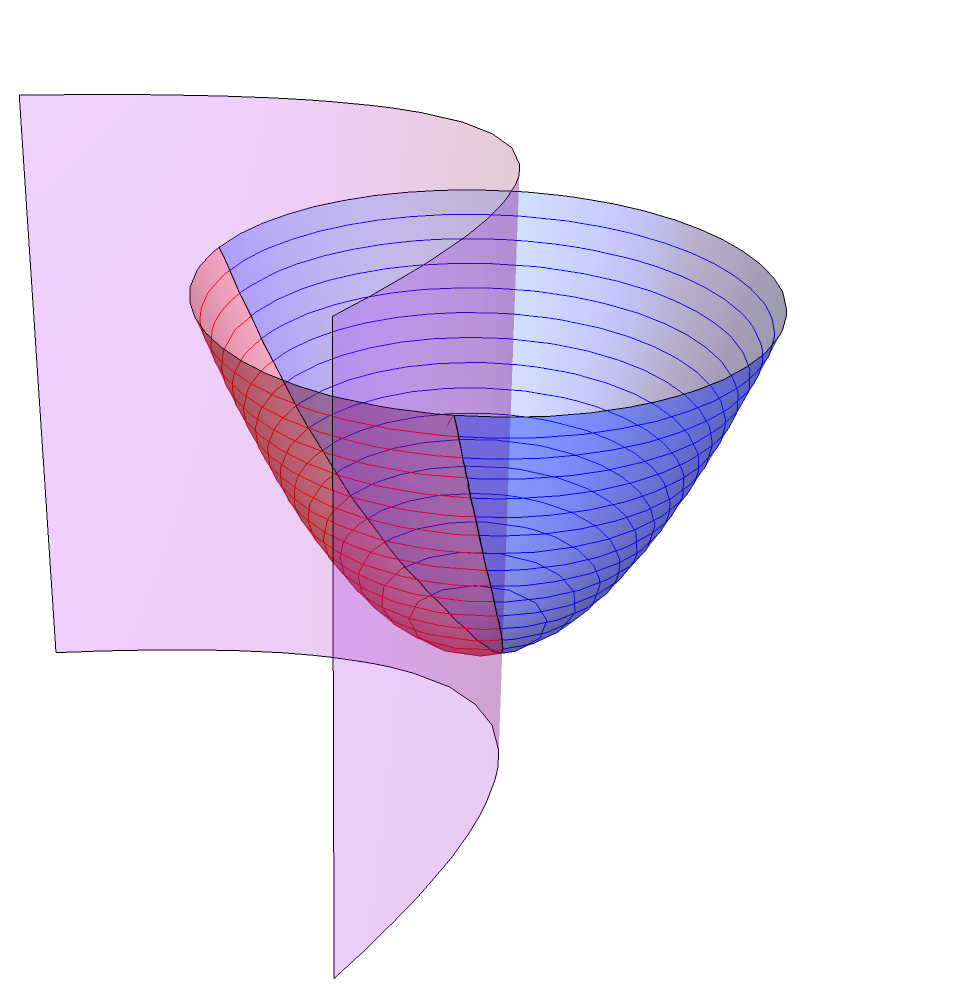}
\put(55,15){$\Sigma$}
  \put(65,80){$\mathcal{S}$}
\end{overpic}}
\caption{Submanifold of $S$ foliated by crossing periodic orbits for $f(x,y,z) \in \{y,y-z,  y-x^2\}$.}\label{fig5}
\end{center}
\end{figure}
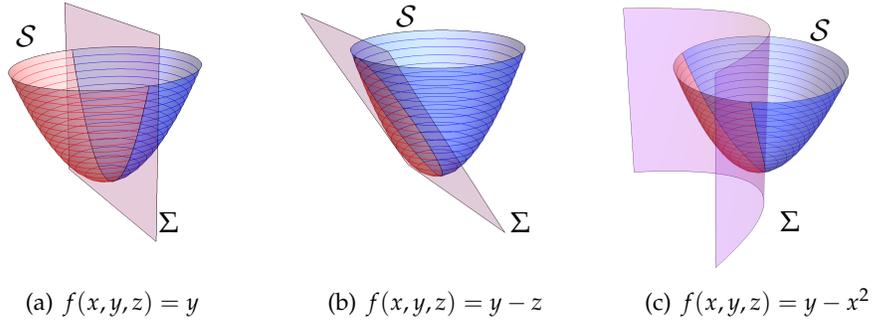

\begin{proposition}\label{cpara}
For $f(x,y,z) \in \{y,y-z,  y-x^2\}$, the unperturbed differential system \eqref{example2} has a invariant manifold $\mathcal{T}^*$ of codimention one of $\R^3$ foliated by crossing periodic orbits  such that $\mathcal{T}^*\subset \mathcal{S}$, see Figure \ref{fig5}. Then we have the following statement,
\begin{enumerate}
\item For $f(x,y,z)=y$, $\Sigma \cap \mathcal{T}^* =\{(x,0,x^2)\in \R^3: x\in \R\},$
\item For $f(x,y,z)=y-z$, $\Sigma \cap \mathcal{T}^*=\{(x,y,y)\in \R^3:  (2x)^2 + (2y - 1)^2=1 \},$
\item For $f(x,y,z)=y-x^2$, $\Sigma \cap \mathcal{T}^*=\{(x,x^2,x^2+x^4)\in \R^3: x\in \R \}.$
\end{enumerate}
\end{proposition}

\begin{theorem}\label{para}
For $f(x,y,z) \in \{y,y-z,  y-x^2\}$and  $\varepsilon\neq 0$ sufficiently small,  the differential system \eqref{example1} admits at least $\ell_h$  limit cycle converging, when $\varepsilon\rightarrow0$, to some of the crossing periodic orbits contained in $\mathcal{S}$, where
\begin{enumerate}
\item$\ell_h=4$,  if $f(x,y,z)=y$,
\item $\ell_h=6$,  if $f(x,y,z)=y-z$,
\item $\ell_h=8$,  if $f(x,y,z)=y-x^2$.
\end{enumerate}
\end{theorem}

The remainder of the paper is organized as follows. In Section \ref{SecT1}, we provide the proof of Theorem \ref{thm}. Section \ref{Ex}  contains the proofs of Theorems \ref{p1}, \ref{quadratic}, and \ref{para}.

   \section{Proof of Theorem \ref{thm}} \label{SecT1}

In the next lemma, we present at particular case of the Lyapunov–Schmidt reduction method that we shall need for proving the
main results of this paper. For more details, see \cite{JN3}.

\begin{lemma}\label{LS}
 Assume that $m \geq n$ are positive integers. Let $\mathcal{D}$ and $\mathcal{V}$ be open subsets of $\R^m$ and $\R^n$, respectively. Let  ${\bf g}_0, {\bf g}_1$ and $v:\overline{ \mathcal{V}} \rightarrow \R^{m-n}$ be $\mathcal{C}^2$ functions, consider 
${\bf g} : \mathcal{D} \times (-\varepsilon_0, \varepsilon_0) \rightarrow \R^m$ as
$${\bf g} (z, \varepsilon) = {\bf g} _0(z) +\varepsilon {\bf g} _1(z)+\mathcal{O}(\varepsilon^2),$$
and take $ \mathcal{Z} = \{z_u = (u, v(u)) : u \in \overline{\mathcal{V}} \}\subset \mathcal{D}$. We denote by $\Gamma_u$ the upper right corner $n \times (m -n)$ matrix of $D_z{\bf g} _0(z_u)$, and by $\Delta_u$, the lower right corner $(m -n) \times (m-n)$ matrix of $D_z\bf{g}_0(z_u)$. Assume that for each $z_u\in \mathcal{Z}$, $\det(\Delta_u)\neq 0$ and ${\bf g} _0(z_u)=0$. We define the bifrucation function the functions $f_1 :\overline{ \mathcal{V}} \rightarrow \R^n$ as
\begin{equation}\label{fLS}
f_1(u)=-\Gamma_u (\Delta_u)^{-1} \pi_2(g_1(z_u)))+\pi_1(g_1(z_u)).
\end{equation}
 If there exists $u^{*} \in \mathcal{V}$ with $f_1(u^{*})=0$ and $\det(D f_1(u^*)) \neq 0$, then there exists $u_\varepsilon$ such
that  $g(z_{u_\varepsilon}, \varepsilon)=0$ and $z_{u_{\varepsilon}} \rightarrow z_{u^{*}}$ when $\varepsilon \rightarrow 0$.
\end{lemma}

Let $\varphi^{\pm}(t, x,y,z, \varepsilon)$ the solution of 
 \begin{equation*} 
        \left(\dot{x}, \dot{y}, \dot{z}\right)^T=X_0^\pm(x,y,z)+\varepsilon X_1^\pm (x,y,z)+\varepsilon^{2}R^\pm (x,y,z,\varepsilon), 
    \end{equation*}
 with $\varphi^{\pm}(0, x,y,z, \varepsilon)=(x,y,z)$. Then, $\varphi^{\pm}(t, x,y,z, 0)=\varphi^{\pm}_0(t, x,y,z)$. From hypothesis {\bf (H1)}, the solution of  \eqref{s0}, for $\varepsilon=0$, starting at $(u, v(u), g(u, v(u)))$ reaches transversely the set of discontinuity. Therefore, for a neighborhood $\mathcal{W} \subset D_1\times D_2$ of $\mathcal{U}$ and $|\varepsilon|\neq 0$ small enough there exists a
time $t^+(x, y, \varepsilon) > 0$ with $t^+(u,v(u),0)=\tau^+(u)$ (resp. $t^-(x, y, \varepsilon) < 0$ with $t^-(u,v(u),0)=\tau^-(u)$) such that an trajectory of  \eqref{s0} starting in $(x, y, g(x,y))\in\text{Graphs}(g|_{\mathcal{W}})$ returns, forward in time (resp. backward in time), to $\Sigma$, that is
$$( h\circ \varphi^{\pm})(t^{\pm}(x,y,\varepsilon),x,y,g(x,y),\varepsilon)=0.$$
The next propositions provide the expressions of the partial derivates of $\varphi^{\pm}(t,x,y,z, \varepsilon)$ and $t^{\pm}(x,y,\varepsilon)$ at $\varepsilon=0$, in terms of the  solutions of unperturbed differential system \eqref{s0}.

\begin{proposition}\label{derflux} Let $(x,y,z)\in D$. Then
$$\dfrac{ \partial\varphi^{\pm}}{\partial \varepsilon}(t,x,y,z,0)= \omega^{\pm} (t,x,y,z)$$
where
 $$\omega^{\pm} (t,x,y,z)=D_{(x,y,z)} \varphi^{\pm}_0(t,x,y,z)\int_{0}^t (D_{(x,y,z)} \varphi^{\pm}_0(s,x,y,z))^{-1} X_1^\pm ( \varphi^{\pm}_0(s,x,y,z))ds.$$
\end{proposition}
\begin{proof}
It is worth noting that the functions $w_1^\pm (t)=\frac{\partial\varphi^{\pm}}{\partial \varepsilon}(t,x,y,z,0)$ and $w^\pm_0(t)=\omega^\pm(t,x,y,z)$ are solutions of the initial value problem
\begin{equation*}
\left\{
\begin{aligned}
\dot{w}&=D_{(x,y,z)}X_0^{\pm}(\varphi^{\pm}(t,x,y,z,0)) w+X_1^{\pm}(\varphi^{\pm}(t,x,y,z,0)),\\
w(0)&=0.
\end{aligned}
\right.
\end{equation*}
The proof of this proposition follows by applying the Existence and Uniqueness Theorem.
\end{proof}

\begin{proposition}\label{dertem} 
Let $(x,y)\in \mathcal{W}$. Then,
\begin{equation*}
\begin{aligned}
\frac{\partial t^{\pm}}{\partial \varepsilon}(x,y,0)&=-\frac{\nabla h( \varphi^{\pm}_0(t^{\pm}(x,y,0),x,y,g(x,y))) \cdot \omega^{\pm}(t^{\pm}(x,y,0),x,y,g(x,y))}{X_0^{\pm}h( \varphi_0^{\pm}(t^{\pm}(x,y,0),x,y,g(x,y))) },\\
\\
\frac{\partial t^{\pm}}{\partial y}(x,y,0)&=-\frac{\nabla h( \varphi^{\pm}_0(t^{\pm}(x,y,0),x,y,g(x,y)))\cdot Y^{\pm}(t^{\pm}(x,y,0),x,y,g(x,y))}{X_0^{\pm}h( \varphi_0^{\pm}(t^{\pm}(x,y,0),x,y,g(x,y))) }
\end{aligned}
\end{equation*}
where
\[
\omega^\pm(t,x,y,z)= D_{(x,y,z)} \varphi^{\pm}_0(t,x,y,z)\int_{0}^t (D_{(x,y,z)} \varphi^{\pm}_0(s,x,y,z))^{-1}\cdot X_1^\pm (\varphi^{\pm}_0(s,x,y,z))ds
\]
and
\[
Y^\pm(t,x,y,z) = \frac{\partial \varphi_0^\pm}{\partial y}(t,x,y,z) + \frac{\partial \varphi_0^\pm}{\partial z}(t,x,y,z) \cdot \frac{\partial g}{\partial y}(x,y).
\]
\end{proposition}

\begin{proof}
Computing the derivative in the variable $\varepsilon$ at $\varepsilon=0$ of both sides of the equality $h( \varphi^{\pm}(t^{\pm}(x,y,\varepsilon),x,y,g(x,y),\varepsilon))=0$, we obtain
\begin{equation*}
\begin{aligned}
0 &= X^{\pm}_0h( \varphi^{\pm}(t^{\pm}(x,y,0),x,y,g(x,y),0)) \cdot \frac{\partial t^{\pm}}{\partial \varepsilon}(x,y,0) \\
&\quad +  \nabla h( \varphi^{\pm}(t^{\pm}(x,y,0),x,y,g(x,y),0)) \cdot \frac{\partial \varphi^{\pm}}{\partial \varepsilon}(t^{\pm}(x,y,0),x,y,g(x,y),0).
\end{aligned}
\end{equation*}
By Proposition \ref{derflux}, it follows that
\[
\frac{\partial t^{\pm}}{\partial \varepsilon}(x,y,0)=-\frac{\nabla h( \varphi^{\pm}_0(t^{\pm}(x,y,0),x,y,g(x,y))) \cdot \omega^{\pm}(t^{\pm}(x,y,0),x,y,g(x,y))}{X_0^{\pm}h( \varphi_0^{\pm}(t^{\pm}(x,y,0),x,y,g(x,y))) }
\]
with
\[
\omega^\pm(t,x,y,z)= D_{(x,y,z)} \varphi^{\pm}_0(t,x,y,z)\int_{0}^t (D_{(x,y,z)} \varphi^{\pm}_0(s,x,y,z))^{-1}\cdot X_1^\pm (\varphi^{\pm}_0(s,x,y,z))ds.
\]
In the same way we get that
\[
\frac{\partial t^{\pm}}{\partial y}(x,y,0)=-\frac{\nabla h( \varphi^{\pm}_0(t^{\pm}(x,y,0),x,y,g(x,y)))\cdot Y^{\pm}(t^{\pm}(x,y,0),x,y,g(x,y))}{X_0^{\pm}h( \varphi_0^{\pm}(t^{\pm}(x,y,0),x,y,g(x,y))) }
\]
where 
\[
Y^\pm(t,x,y,z) = \frac{\partial \varphi_0^\pm}{\partial y}(t,x,y,z) + \frac{\partial \varphi_0^\pm}{\partial z}(t,x,y,z) \cdot \frac{\partial g}{\partial y}(x,y).
\]
This concludes the proof of the proposition.
\end{proof}

\subsection{Proof of  Theorem \ref{thm}}
We consider the displacement function $\Delta: \mathcal{W} \times (-\varepsilon_0, \varepsilon_0) \rightarrow \mathbb{R}^n \times \mathbb{R}^{n-m} \times \mathbb{R}$ given by
\[
\Delta(x,y, \varepsilon) = \varphi^{+}(t^+(x,y,\varepsilon), x,y,g(x,y), \varepsilon) -  \varphi^{-}(t^-(x,y,\varepsilon), x,y,g(x,y), \varepsilon).
\]
Notice that $\pi_3(\Delta(x,y,\varepsilon)) = 0$, and $\Delta(x,y,0)$ vanishes on the submanifold $\mathcal{U}$, meaning $\Delta(u, v(u),0) = 0$ for $u \in \overline{\mathcal{V}}$. Utilizing Propositions \ref{derflux} and \ref{dertem}, we determine that the derivatives of $\Delta$ with respect to $\varepsilon$ and $y$, at $(u,v(u),0)$, are as follows:
\[
\dfrac{\partial \Delta }{\partial \varepsilon}(u,v(u),0) =\boldsymbol{\alpha}(u)\quad \text{and}\quad
\dfrac{\partial \Delta }{\partial y}(u,v(u),0) = \boldsymbol{\beta}(u),
\]
where $\boldsymbol{\alpha}(u) =\boldsymbol{\alpha}^+(u) - \boldsymbol{\alpha}^-(u)$ and $\boldsymbol{\beta}(u) =\boldsymbol{\beta}^+(u) - \boldsymbol{\beta}^-(u)$ with
\[
\begin{aligned}
\boldsymbol{\alpha}^\pm(u) &= \omega^\pm \circ P^\pm(u) - X_0^\pm(\varphi_0^\pm \circ P^\pm(u)) \cdot \frac{\nabla h(\varphi_0^\pm \circ P^\pm(u)) \cdot (\omega_0^\pm \circ P^\pm(u))}{X_0^\pm h(\varphi_0^\pm \circ P^\pm(u))} \\
&= \left( \text{Id}_{m+1} - \frac{X_0^\pm(\varphi_0^\pm \circ P^\pm(u)) \cdot \nabla h(\varphi_0^\pm \circ P^\pm(u))}{X_0^\pm h(\varphi_0^\pm \circ P^\pm(u))} \right) \cdot \omega^\pm \circ P^\pm(u),
\end{aligned}
\]
and
\[
\begin{aligned}
\boldsymbol{\beta}^\pm(u) &= Y^\pm \circ P^\pm(u) - \frac{X_0^\pm(\varphi_0^\pm \circ P^\pm(u)) \cdot \nabla h(\varphi_0^\pm \circ P^\pm(u)) \cdot (Y^\pm \circ P^\pm(u))}{X_0^\pm(\varphi_0^\pm \circ P^\pm(u))} \\
&= \left( \text{Id}_{m+1} - \frac{X_0^\pm(\varphi_0^\pm \circ P^\pm(u)) \cdot \nabla h(\varphi_0^\pm \circ P^\pm(u))}{X_0^\pm h(\varphi_0^\pm \circ P^\pm(u))} \right) \cdot Y^\pm \circ P^\pm(u),
\end{aligned}
\]
where $P^{\pm}(u) = (\tau^{\pm}(u),u,v(u),g(u,v(u)))$, $\text{Id}_{\ell}$ is the $\ell\times\ell$ identity matrix, and
\[
Y^\pm(t,x,y,z) = \frac{\partial \varphi_0^\pm}{\partial y}(t,x,y,z) + \frac{\partial \varphi_0^\pm}{\partial z}(t,x,y,z) \cdot \frac{\partial g}{\partial y}(x,y).
\]

Now, we define a proper function to apply Lemma 6. Let ${\bf g}: \mathcal{W} \times (-\varepsilon_0, \varepsilon_0) \rightarrow \mathbb{R}^n \times \mathbb{R}^{n-m}$ be given by
\begin{equation}
\begin{aligned}
{\bf g}(x,y, \varepsilon) =& \left(\pi_1 \circ \Delta(x,y,\varepsilon), \pi_2 \circ \Delta(x,y,\varepsilon) \right)\\
=&\left(\Pi_1 \cdot \Delta(x,y,\varepsilon), \Pi_2 \cdot \Delta(x,y,\varepsilon) \right)
\end{aligned}
\end{equation}
Expanding ${\bf g}(x,y, \varepsilon)$ around $\varepsilon=0$, and using Propositions \ref{derflux} and \ref{dertem}, we get
\[
{\bf g}(x, y, \varepsilon) = {\bf g}_0(x,y) + \varepsilon {\bf g}_1(x,y) + \mathcal{O}(\varepsilon^2),
\]
where ${\bf g}_0(x,y) = {\bf g}(x,y, 0)$ and 
\[
\begin{aligned}
{\bf g}_1(x,y) &= \left( \frac{\partial (\pi_1 \circ \Delta)}{\partial \varepsilon}(x,y,0), \frac{\partial (\pi_2 \circ \Delta)}{\partial \varepsilon}(x,y,0) \right) \\
&= \left( \Pi_1\cdot \frac{\partial \Delta }{\partial \varepsilon}(x,y,0), \Pi_2\cdot \frac{\partial \Delta }{\partial \varepsilon}(x,y,0) \right).
\end{aligned}
\]

Notice that ${\bf g}_0(x,y)$ vanishes on the submanifold $\mathcal{U}$, i.e., ${\bf g}_0(u, v(u))=0$ for $u \in \overline{\mathcal{V}}$. Additionally,
\[
\begin{aligned}
{\bf g}_1(u,v(u)) &= \left( \Pi_1\cdot \boldsymbol{\alpha}(u), \Pi_2\cdot  \boldsymbol{\alpha}(u)\right),
\end{aligned}
\]
and
\[
\begin{aligned}
D_{(x,y)}{\bf g}_0(u, v(u)) &= \left(
    \begin{array}{ll}
        \dfrac{\partial (\pi_1\circ \Delta)}{\partial x}(u,v(u),0) & \dfrac{\partial (\pi_1\circ \Delta)}{\partial y}(u,v(u),0) \\
        \dfrac{\partial (\pi_2\circ \Delta)}{\partial x}(u,v(u),0) & \dfrac{\partial (\pi_2\circ \Delta)}{\partial y}(u,v(u),0)
    \end{array}
\right) \\
&= \left(
    \begin{array}{ll}
        \Pi_1 \cdot \dfrac{\partial  \Delta}{\partial x}(u,v(u),0) & \Pi_1 \cdot \dfrac{\partial  \Delta}{\partial y}(u,v(u),0)\\
        \Pi_2 \cdot \dfrac{\partial  \Delta}{\partial x}(u,v(u),0) & \Pi_2 \cdot \dfrac{\partial  \Delta}{\partial y}(u,v(u),0)
    \end{array}
\right),
\end{aligned}
\]
with 
\[
\Pi_i \cdot \frac{\partial  \Delta}{\partial y}(u,v(u),0) = \Pi_i \cdot \boldsymbol{\beta}(u)
\]
for $i=1,2$. Following the notation of Lemma \ref{LS}, we get that
\[
\Gamma_u =  \Pi_1 \cdot \boldsymbol{\beta}(u)
\quad \text{and}\quad \Delta_u =  \Pi_2 \cdot \boldsymbol{\beta}(u).
\]
Hence, if $\text{det}(\Delta_u) = \text{det}\left(  \Pi_2 \cdot \boldsymbol{\beta}(u)\right) \neq 0$ for every $u \in \overline{\mathcal{V}}$, it follows that the bifurcation function \eqref{fLS} becomes
\[
\begin{aligned}
f(u)=\mathcal{M}(u) =& \pi_1(g_1(z_u)) - \Gamma_u (\Delta_u)^{-1} \pi_2(g_1(z_u)) \\
=&\Pi_1\cdot \boldsymbol{\alpha}(u)  - \Pi_1 \cdot \boldsymbol{\beta}(u) \cdot \left( \Pi_2 \cdot \boldsymbol{\beta}(u) \right)^{-1} \cdot \Pi_2\cdot \boldsymbol{\alpha}(u) \\
=& \Pi_1\cdot \left( \text{Id}_{m+1}  - \boldsymbol{\beta}(u) \cdot \left( \Pi_2 \cdot \boldsymbol{\beta}(u) \right)^{-1} \cdot \Pi_2\right)\cdot \boldsymbol{\alpha}(u) \\
\end{aligned}
\]
The proof of this Theorem follows by applying Lemma \ref{LS}.

\section{Proofs of Examples} \label{Ex}
In the sequence, we present some  concepts  and a classical result about Extended Complete Chebyshev system (ECT-system) need to proveTheorems \ref{p1}, \ref{quadratic} and \ref{para}.\\

Consider a closed interval $I$ in $\R$. We define an Extended Chebyshev system (ET-system), denoted $\F = [f_0, f_1,\ldots, f_n]$, on $I$ as an ordered set of real functions such that any nontrivial linear combination of its elements has at most $n$ zeros, counting multiplicities. If $[f_0, f_1,\ldots, f_k]$ forms an ET-system for all $k\in{0, \ldots,n}$, then $\F$ is termed an Extended Complete Chebyshev system (ECT-system) on $I$. The verification that $\F$ constitutes an ECT-system on $I$ requires demonstrating that $W_k(x)=W[f_0, \ldots, f_k](x) \neq 0$ on $I$ for all $k \in{0, \ldots, n}$. Here, $W[f_0, \ldots, f_k](x)$ represents the Wronskian of $\F$ with respect to $t$.
\[
W[f_0,\ldots f_n](t)(x)=\det \left( \begin{array}{ccc}
u_0(x)& \ldots& u_s(x)\\
u'_0(x)& \ldots & u'_s(x)\\
\vdots& & \vdots\\
u_0^{(s)}(x) &&  u_s^{(s)}(x) 
\end{array}\right).
\]
 Further details can be found in \cite{KASTU1966}. The next theorem, proved in \cite{KASTU1966}, is classical  result about  ECT-systems

\begin{theorem}\label{t1}
Let $\F=[u_0, u_1, ..., u_n]$ be an ECT-system on a closed interval  $[a, b]$. Then, the number of isolated zeros for every element of $\mbox{Span}(\F)$ does not exceed $n$. Moreover, for each configuration of $m \leq n$ zeros, taking into account their multiplicity, there exists $F\in \mbox{Span}(\F)$ with this configuration of zeros.
\end{theorem}

\subsection{Proof of Theorem \ref{p1}}\label{sec:proofp1}
 Let $f(x,y,z)=z$. Considering the change of coordinates 
\begin{equation*}\label{coor}
(x,y,z)\rightarrow(y,x,z)
\end{equation*}
the differential system \eqref{example1} for $i\in\{a,b\}$
becomes
\begin{equation*}\label{example11}
\left(\dot{x}, \dot{y}, \dot{z}\right)^T=\left\{
\begin{array}{ll}
\tilde{X}_{i,0}^+(x,y,z)+\varepsilon\tilde{ X}_1^-(x,y,z)& z>0,\\
\tilde{X}_{i,0}^-(x,y,z)+\varepsilon \tilde{X}_1^-(x,y,z)& z<0,
\end{array}\right.
\end{equation*}
where
\begin{equation*}
\tilde{X}_{a,0}^{\pm}(x,y,z)=
\left(
\begin{array}{c}
-z\mp 1\\
y\\
x+y
\end{array}
\right), \quad
\tilde{X}_{0,b}^{\pm}(x,y,z)= \left(
\begin{array}{c}
z\mp 1\\
y\\
x+y
\end{array}\right), 
\end{equation*}
and
\begin{equation*}
\tilde{X}_1^{\pm}(x,y,z)=
\left(
\begin{array}{l}
  \beta_0^++\beta_2^{\pm} x+\beta_1^{\pm} y+\beta_3^{\pm} z,\\
\alpha_0^++\alpha_2^{\pm} x+\alpha_1^{\pm} y+\alpha_3^{\pm} z,\\
\kappa_0^++\kappa_1^{\pm} y+\kappa_2^{\pm} x+\kappa_3^{\pm} z
\end{array}\right).
\end{equation*}\
Let $i=a$. Computing the solution of system  $\left(\dot{x}, \dot{y}, \dot{z}\right)^T=\tilde{X}_{a,0}(x,y,z)$ with initial condition $(x,y,z)$, we obtain
\begin{equation*}\varphi^{\pm}_0(t,x,y,z)=
\left(
\begin{array}{c}
\frac{1}{2} \left(\cos (t) (2 x+y)+\sin (t) (y-2 z\mp2)-e^t y\right)\\
e^t y\\
\frac{1}{2} \left(\sin (t) (2 x+y)+\cos (t) (2 z-y\pm2)+e^t y\mp2\right)
\end{array}\right).
\end{equation*}

Then, the submanifold of initial conditions $\Sigma$ whose orbits are periodic and each of them reaches transversally of $\Sigma$ is $\mathcal{C}= \{(x,0,0)\in \R^3: x\neq0\}$. In this case, the return-times are $\tau^\pm(x)=\pm2\arctan(x)$. In order to apply Theorem \ref{thm} we need to write some submanifold of $\mathcal{C}$ as 
$$\mathcal{Z}=\text{Graph}(g|_{\mathcal{U}}),$$ 
with $\mathcal{U}=\{(u, v(u)): u\in\overline{\mathcal{V}}\}$
where $\mathcal{V}$ an open bounded subset of $\mathcal{\R}$ and $v:\overline{\mathcal{V}} \rightarrow \mathcal{\R}$ a $\mathcal{C}^{2}$ function. Let $\mathcal{V}=(a_0,b_0)$, with $0<a_0<b_0$, with $a_0$ sufficiently engough small, and $v(u)=0$. Therefore, using the formulae \eqref{Mel} and \eqref{beta} we get that $
\boldsymbol{\beta}(u)=2 \sinh \left(2 \tan ^{-1}(u)\right)\neq 0$
and
\begin{equation}\label{Mel2}
2u\mathcal{M}(u)=2 L_{0} f_0(u)-L_{1} f_1(u)+L_{2} f_2(u),
\end{equation}
where $f_0(u)=u$, $f_1(u)=\left(u^2+1\right) \tan ^{-1}(u)$, $f_2(u)= \tanh \left(\tan ^{-1}(u)\right)$ and
\begin{equation*}
\begin{array}{ll}
L_{0}=&-\alpha_0^-+\alpha_0^+-\alpha_3^--\alpha_3^++\beta_2^-+\beta_2^++2 \kappa_0^--2 \kappa_0^++\kappa_3^-+\kappa_3^+\\
L_{1}=&\alpha_2^-+\alpha_2^+-\alpha_3^--\alpha_3^++2 (\beta_2^-+\beta_2^+)+2 \kappa_3^-+2 \kappa_3^+,
\\L_{2}=&2 \alpha_0^--2 \alpha_0^++\alpha_2^-+\alpha_2^++\alpha_3^-+\alpha_3^+.
\end{array}
\end{equation*}
 Straightforward computations give us the the parameter vector $(L_0, L_1, L_2) \in \R^3$ depends on $\alpha_i^+, \beta_i^{\pm}$ and  $\kappa_i^\pm$ in a surjective way.
A lower bound of the number of simple zeros that \eqref{Mel2} can have on $(a_0, b_0)$ follows by studying the Wronskians
of the ordered set $\F_a=[f_0, f_1, f_2]$ at $u=0$. For that, we compute the Taylor series of functions $f_1$ an $f_2$ around $u = 0$,
$$f_1(u)=u+\dfrac{2 u^3}{3}-\dfrac{2 u^5}{15}+\mathcal{O}(u^6)\quad \text{ and } \quad f_2(u)=u-\frac{2 u^3}{3}+\frac{2 u^5}{3}+\mathcal{O}(u^6).$$
Then 
\begin{equation*}
\begin{array}{ll}
W_0(u)=u,\\
W_1(u)= \dfrac{4}{3} u^5 +O\left(u^6\right),\\
W_2(u)= \dfrac{256 u^6}{45}+O\left(u^7\right).\\
\end{array}
\end{equation*}
This implies that there exists $r_0>0$ sufficiently small, such that the Wronskian $W_1, W_2$ and $W_3$ do not vanishes on $(0, r_0)$ and, consequently, $\F_a$ is an ECT-system on $(0, r_0)$. The proof of  proposition for $i=a$,
follows by applying  Theorems \ref{t1} and \ref{thm}.

Now, consider $i=b$. Then,  the solution of the system  $\left(\dot{x}, \dot{y}, \dot{z}\right)^T=\tilde{X}_{b,0}(x,y,z)$ with initial condition $(x,y,z)$ is
\begin{equation*}\varphi^{\pm}_0(t,x,y,z)=
\left(
\begin{array}{c}
\dfrac{1}{2} (\cosh (t) (t y+2 x)+\sinh (t) ((t-1) y+2 (z\mp1)))\\
e^t y\\
\dfrac{1}{2} (\sinh (t) (t y+2 x+y)+\cosh (t) (t y+2 z\mp2)\pm2).
\end{array}\right)
\end{equation*}
In this case, $\mathcal{C}= \{(x,0,0)\in \R^3: 0<|x|<1\}$ and the return times are 
 $$\tau^\pm(x)=\log\left(\dfrac{1\pm x}{1\mp x}\right).$$
Let $\mathcal{Z}=\text{Graph}(g|_{\mathcal{U}})$ 
with $\mathcal{U}=\{(u, v(u)): u\in\mathcal{V}\}$ where $\mathcal{V}=(a_1, b_1)$,  with $a_1$ sufficiently enough small, and $v(u)=0$. By the formulae \eqref{Mel} and \eqref{beta} we get that 
$\boldsymbol{\beta}(u)=-\dfrac{4 u}{-1 + u^2}$ and
\begin{equation*}
8u \mathcal{M}(u)=2 L_{0} f_0(u)-2 L_{1}f_1(u)+4 L_{2} f_2(u),
\end{equation*}
where $f_0(u)=u$, $f_1(u)=\left(u^2-1\right) \tanh ^{-1}(u)$, $f_2(u)=u \left(u^2-1\right) \tanh ^{-1}(u)^2$ and
\begin{equation*}
\begin{array}{ll}
L_{0}=&-4 \alpha_0^-+4 \alpha_0^+-\alpha_2^--\alpha_2^++3 \alpha_3^-+3 \alpha_3^+-4 (\beta_2^-+\beta_2^+)+8 \kappa_0^--8 \kappa_0^+-4 \kappa_3^--4 \kappa_3^+,\\
L_{1}=&4 \alpha_0^--4 \alpha_0^++\alpha_2^-+\alpha_2^+-3 \alpha_3^--3 \alpha_3^++4 (\beta_2^-+\beta_2^+)+4 \kappa_3^-+4 \kappa_3^+,
\\L_{2}=&\alpha_2^-+\alpha_2^++\alpha_3^-+\alpha_3^+.
\end{array}
\end{equation*}
It is easy to seen that the parameter vector $(L_0, L_1, L_2) \in \R^3$ depends on $\alpha_i^+, \beta_i^{\pm} $ and $k_i^\pm$ in a surjective way. Moreover, in a neighborhood of the origin
$$f_1(u)=- u+\frac{2 u^3}{3}+\frac{2 u^5}{15}+O\left(u^6\right) \quad \text{and} \quad f_2(u)=- u^3+\frac{u^5}{3}+O\left(u^6\right).$$
Hence, the wronskian of the ordered set of functions $\F_2=[f_0, f_1, f_2]$ around at $u=0$ are
\begin{equation*}
\begin{array}{ll}
W_0(u)=u,\vspace{0.1cm}\\
W_1(u)= \dfrac{4}{15} u^3 +\mathcal{O}(u^6),\vspace{0.1cm}\\
W_2(u)= \dfrac{256}{45} u^6+\mathcal{O}(u^7).\\
\end{array}
\end{equation*}
 It concludes the proof of proposition for $i=b.$

\subsection{Proof of Theorem \ref{quadratic}}\label{sec:proofquadratic}
The proof that the unperturbed differential system \eqref{example1} for $i=a$ and $f(x,y,z)$ has an invariant plane at $\mathcal{S}={(x,y,z)\in \R^3: x=0}$ containing a period annulus foliated by crossing periodic orbits is straightforward and will be omitted. 
 Considering the change of coordinates \eqref{coor}, $\Sigma=\{(x,y,z): f(x,y,z)=-d x^2 - c x y - y^2 + z\}$ becomes
$$\Sigma=\{(x,y,z): z=g(x,y)\},$$
 with  $g(x,y)=x^2 + c x y + d y^2$. Thus, the submanifold of initial conditions of $\Sigma$ whose orbits are periodic and each of them reaches transversally of $\Sigma$ is $\mathcal{C}=
\{(x,0,x^2)\in \R^3: x>1\}$. Let $\mathcal{Z}=\text{Graph}(g|_{\mathcal{U}})$ 
with $$\mathcal{U}=\{(u, v(u)): u\in\overline{\mathcal{V}}\},$$
where  $\mathcal{V}=(a_2,b_2)$, with $2<a_2<b_2$, and $v(u)=0$. Then
$\tau^+(u)=\tan ^{-1}\left(\frac{2 \left(u^3+u\right)}{u^4+u^2+1}\right)>0$ and $$\tau^-(u)=\tan ^{-1}\left(\dfrac{2 u \left(u^2-1\right)}{u^4-3 u^2+1}\right)-2 \pi<0.$$
By the formulae \eqref{Mel} and \eqref{beta}, it follows that
$\boldsymbol{\beta}(u)=e^{\text{$\tau^+(u) $}}-e^{\tau^-(u)}\neq 0$
\begin{equation*}
4 \pi  u \left(4 u^4+4 u^2-3\right) \left(e^{\tau^+(u)}-e^{\tau^-(u)}\right) \mathcal{M}(u)= \sum_{i=0}^7 L_i f_i(u).
\end{equation*}
where the parameter vector $(L_0, L_1, L_2, L_3, L_4, L_5, L_6, L_7) \in \R^8$ depends on $\alpha_i^+, \beta_i^{\pm}$ and $k_i^\pm$ in a surjective way. And
\begin{equation*}
\begin{array}{ll}
f_0(u)=&e^{\tau^-(u)} \left(e^{\tau^+(u)} (2 (c-1) u+1)-1\right)-2 (c-1) u-e^{\tau^+(u)}+1,\\
f_1(u)=&u \left(e^{\tau^+(u)}-e^{\tau^-(u)}\right),\\
f_2(u)=&  u^2 \left(e^{\tau^-(u)} \left(e^{\tau^+(u)} (2 (c-1) u+1)-1\right)-2 (c-1) u-e^{\tau^+(u)}+1\right),\\
f_3(u)=&  u^3 \left(e^{\tau^+(u)}-e^{\tau^-(u)}\right),\\
f_4(u)=&\tau^-(u) \left(2 u^6+u^4-u^2+3\right) \left(e^{\tau^+(u)}-e^{\tau^-(u)}\right)\\
&-2 \pi  \left(u^2-3\right) \left(e^{\tau^-(u)} \left(e^{\tau^+(u)} (2 (c-1) u+1)-1\right)-2 c u-e^{\tau^+(u)}+2 u+1\right),\\
f_5(u)=&  u^5 \left(e^{\tau^+(u)}-e^{\tau^-(u)}\right),\\
f_6(u)=& \tau^+(u) \left(2 u^2-1\right) \left(u^4+3 u^2+1\right) \left(e^{\tau^+(u)}-e^{\tau^-(u)}\right),\\
f_7(u)=&\pi  \left(2 u (2 (c-1) u+1)^2 e^{\tau^+(u)+\tau^-(u)}-2 u (1-2 (c-1) u)^2\right).
\end{array}
\end{equation*}
Computing the wronskians of $\F_3=[f_0,f_1,f_2,f_3, f_4, f_5, f_6, f_7]$ around at $u=2$, we get that $$
W_i(u)=k_i+\mathcal{O}(u-2),$$ with $k_i\neq0$, for $i=0, \ldots,7.$ Then,  the proof of  proposition
follows by applying Theorems \eqref{thm} and \eqref{t1}.

\subsection{Proof of Theorem \ref{para}}\label{sec:proofpara}
Performing the change of coordinates $(x,y,z)\rightarrow(x,z,y)$, the differential system \eqref{example2}
becomes
\begin{equation*}\label{example22}
\left(\dot{x}, \dot{y}, \dot{z}\right)^T=X_0(x,y,z)+\epsilon X_1(x,y,z),
\end{equation*}
where
\begin{equation*}
X_{0}(x,y,z)=
\left(
\begin{array}{c}
-z\\
\lambda(x^2+z^2-y)\\
x
\end{array}
\right), \quad
X_{1}^{\pm}(x,y,z)= \left(
\begin{array}{c}
\tilde{P}_1^{\pm}(x,z,y)\\
\tilde{P}_3^{\pm}(x,z,y)\\
\tilde{P}_2^{\pm}(x,z,y)
\end{array}\right), 
\end{equation*}

\begin{equation*}
P_{\ell}(x,y,z)=\left\{
\begin{array}{ll}
P_\ell^+(x,y,z)=\displaystyle \sum_{0\leq i+j+k\leq 2} p^{\ell,+}_{ijk} x^i y^jz^k, & z>g(x,y),\\
\\
P_\ell^-(x,y,z)=\displaystyle \sum_{0\leq i+j+k\leq 2} p^{\ell,-}_{ijk} x^i y^jz^k& z<g(x,y),\\
\end{array}\right.
\end{equation*}
with $g(x,y)\in\{0, y, y^2\}$.
Iin this coordinates, the submanifold of initial conditions $\mathcal{C}\subset\Sigma$ whose orbits are periodic and each of them reaches transversally of $\Sigma$ is
$$\mathcal{C}=
\left\{
\begin{array}{ll}
\{(x,x^2,g(x,x^2))\in \R^3: x>0\}, & g(x,y)=0,\\
\{(x,y,g(x,y))\in \R^3: (2x)^2+(2y-1)^2=1 \},& g(x,y)=y,\\
\{(x,y^2,g(x,y^2))\in: x>0 \},& g(x,y)=y^2.
\end{array}
\right.$$
For each  $g\in\{0,-y, y^2\}$ we consider $\mathcal{Z}$ submanifold of $\mathcal{C}$, that can be written as
$\mathcal{Z}=\text{Graph}(g|_{\mathcal{U}})$
 with $\mathcal{U}=\{(u, v(u)): u\in\overline{\mathcal{V}}\}$
where $\mathcal{V}$ an open bounded subset of $\mathcal{\R}$ and $v:\overline{\mathcal{V}} \rightarrow \mathcal{\R}$ a $\mathcal{C}^{2}$ function.\\

For $g(x,y)=0$, $\mathcal{V}=(a_3,b_3)$ with $0<a_3<b_3$ and $v(u)=u^2$. Therefore $\tau^{\pm}(u)=\pm \pi$. Using the formulae \eqref{Mel} and \eqref{beta}, we get 
$\boldsymbol{\beta}(u)=e^{-\pi  \lambda }-e^{\pi  \lambda }\neq 0$
and
\begin{equation}\label{Mel3}
(e^{-\pi  \lambda }-e^{\pi  \lambda })\mathcal{M}(u)=L_0+L_1u+L_2u^2+L_3 u^3+L_4 u^4.
\end{equation}
where the parameter vector $ (L_0, L_1, L_2, L_3, L_4) \in \R^5$  depends on the original parameters in
a surjective way. By Descartes Theorem, we conclude that the maximum number of simple zeros that \eqref{Mel3} can have is 4. The proof of statement (a) follows by applying Theorems \ref{thm} and \ref{t1}.

For $g(x,y)=y$,  $\mathcal{V}=(a_4, b_4)$, with $0<a_4<b_4<1/4$ and $v(u)=\frac{1}{2} \left(1-\sqrt{1-4 u^2}\right)$. Hence, 
$$\tau_1(u)=\cos ^{-1}\left(-\sqrt{1-4 u^2}\right)$$ and $\tau_{2}(u)=\tau_{1}(u)-2 \pi$.  In this case, $\boldsymbol{\beta}(u)=-\left(e^{2 \pi  \lambda }-1\right) \sqrt{1-4 u^2} e^{-\lambda  \tau_1(u)}\neq 0$
and
$$-\frac{\boldsymbol{\beta}(u)}{12 u} \mathcal{M}(u)=\sum_{i=0}^8L_i f_i(u),$$
where the parameter vector $ (L_0, L_1, L_2, L_3, L_4, L_5, L_6) \in \R^7$  depends on the original parameters in
a surjective way. And $f_0(u)=u$, $f_1(u)=u^2$, $f_2(u)=u^3$,
\begin{equation*}
\begin{array}{lll}
f_3(u)=  u\sqrt{1-4 u^2},& f_5(u)=\sqrt{1-4 u^2}-1,\\
f_4(u) =u^2 \cos ^{-1}\left(-\sqrt{1-4 u^2}\right), & f_6(u)=\left(\sqrt{1-4 u^2}-1\right) \cos ^{-1}\left(-\sqrt{1-4 u^2}\right)\\
\end{array}
\end{equation*}
Computing the wronskians of Wronskians of the ordered set $\F_4=[f_0, f_1, f_2, f_3, f_4, f_5, f_6]$ on $\mathcal{V}$, we have \begin{equation*}
\begin{array}{ll}
W_0(u)=1,\\
W_1(u)= 2 u,\\
W_2(u)= 6 u^2,\\
W_3(u)= -\dfrac{96 u^5}{\left(1-4 u^2\right)^{5/2}},\\
W_4(u)=\dfrac{6144 u^8}{\left(4 u^2-1\right)^5},\\
W_5(u)=\dfrac{36864 u^3 \left(96 u^4+10 \left(8 \sqrt{1-4 u^2}-15\right) u^2-35 \left(\sqrt{1-4 u^2}-1\right)\right)}{\left(1-4 u^2\right)^{15/2}},\\
W_6(u)=-\dfrac{905969664 u^6 \left(2 u^2+\sqrt{1-4 u^2}-1\right)}{\left(1-4 u^2\right)^{21/2}}.\\
\end{array}
\end{equation*}
By straightforward computations, we get that $W_i(u)$, for $i=0,1,2,3,4,5,6$, does not vanish in $\mathcal{V}$. So, from Theorems \ref{thm} and \ref{t1} the proof of statement (b) follows.\\

For $g(x,y)=x^2$, $\mathcal{V}=(a_5, b_5)$, with $0<a_5<b_5<1$, and $v(u)=u^2+u^4$. Hence, 
$$\tau_1(u)=\tan ^{-1}\left(\dfrac{2 u}{u^2-1}\right)$$ and $\tau_{2}(u)=\tau_{1}(u)-2 \pi.$ Thus, $\boldsymbol{\beta}(u)=-\left(e^{2 \pi  \lambda }-1\right) e^{-\lambda  \tau_1(u)}$ and
$$\left(\frac{\left(e^{2 \pi  \lambda }-1\right) e^{-\lambda  \tau_1(u)}}{6 \left(2 u^2+1\right)}\right)\mathcal{M}(u)=\sum_{i=0}^8L_i f_i(u),$$
where where the parameter vector $ (L_0, L_1, L_2, L_3, L_5, L_6, L_7, L_8) \in \R^9$  depends on the original parameters in
a surjective way. And $f_0(u)=1$, $f_1(u)=u^2$, $f_2(u)=u^4$, $f_3(u)=u^6$, $f_4(u)=u^8,$
\begin{equation*}
\begin{array}{lll}
f_5(u)=  u^3 (1 + u^2)^2,& f_7(u)=u (1 - 2 u^4 - u^6),\\
f_6(u)=u^3 (1 + u^2)^2 \tan ^{-1}\left(\dfrac{2 u}{u^2-1}\right), & 
 f_8(u)=(u - 2 u^5 - u^7) \tan ^{-1}\left(\dfrac{2 u}{u^2-1}\right).
\\
\end{array}
\end{equation*} The Wronskians
of the ordered set $\F_5=[f_0, \ldots, f_8]$ on $\mathcal{V}$ are given by 
\begin{equation*}
\begin{array}{ll}
W_0(u)=&1,\\
W_1(u)=& 2 u,\\
W_2(u)=& 16 u^3,\\
W_3(u)=& 786 u^6,\\
W_4(u)=&294912 u^{10},\\
W_5(u)=&-4423680 u^8 (3 - 6 u^2 + 7 u^4),\\
W_6(u)=&-\dfrac{54358179840 u^{12}}{\left(u^2+1\right)^4},\\
W_7(u)=&-\dfrac{4892236185600 u^6 \left(10 u^4-21 u^2+21\right)}{\left(u^2+1\right)^5}\\
W_8(u)=&-\dfrac{601157982486528000 u^{10} \left(2 u^2+15\right)}{\left(u^2+1\right)^{12}}.
\end{array}
\end{equation*}
It can easily be seen that the Wronskians do not vanish on $\mathcal{V}$. Hence, from Theorem \ref{thm} and \ref{t1} the proof  the statement (c) follows.

\section*{Acknowledgments}

DDN is partially supported by S\~{a}o Paulo Research Foundation (FAPESP) grants 2022/09633-5,  2019/10269-3, and 2018/13481-0, and by Conselho Nacional de Desenvolvimento Cient\'{i}fico e Tecnol\'{o}gico (CNPq) grant 309110/2021-1.

\bibliographystyle{abbrv}
\bibliography{references.bib}

\begin{thebibliography}{10}

\bibitem{BBLN19}
J.~L. Bastos, C.~A. Buzzi, J.~Llibre, and D.~D. Novaes.
\newblock Melnikov analysis in nonsmooth differential systems with nonlinear
  switching manifold.
\newblock {\em Journal of Differential Equations}, 267(6):3748 -- 3767, 2019.

\bibitem{BFL}
A.~Buic\u{a}, J.-P. Fran\c{c}oise, and J.~Llibre.
\newblock Periodic solutions of nonlinear periodic differential systems with a
  small parameter.
\newblock {\em Commun. Pure Appl. Anal.}, 6(1):103--111, 2007.

\bibitem{BGL}
A.~Buic\u{a}, J.~Gin\'e, and J.~Llibre.
\newblock A second order analysis of the periodic solutions for nonlinear
  periodic differential systems with a small parameter.
\newblock {\em Phys. D}, 241(5):528--533, 2012.

\bibitem{BGL2}
A.~Buic\u{a}, J.~Gin\'e, and J.~Llibre.
\newblock Periodic solutions for nonlinear differential systems: the second
  order bifurcation function.
\newblock {\em Topol. Methods Nonlinear Anal.}, 43(2):403--419, 2014.

\bibitem{CLN}
M.~R. C\^andido, J.~Llibre, and D.~D. Novaes.
\newblock Persistence of periodic solutions for higher order perturbed
  differential systems via {L}yapunov-{S}chmidt reduction.
\newblock {\em Nonlinearity}, 30(9):3560--3586, 2017.

\bibitem{CLL}
X.~Chen, T.~Li, and J.~Llibre.
\newblock Melnikov functions of arbitrary order for piecewise smooth
  differential systems in {$\Bbb{R}^n$} and applications.
\newblock {\em J. Differential Equations}, 314:340--369, 2022.

\bibitem{GLN}
M.~R.~A. Gouveia, J.~Llibre, D.~D. Novaes, and C.~Pessoa.
\newblock Piecewise smooth dynamical systems: persistence of periodic solutions
  and normal forms.
\newblock {\em J. Differential Equations}, 260(7):6108--6129, 2016.

\bibitem{KASTU1966}
S.~Karlin and W.~Studden.
\newblock Tchebycheff systems: With applications in analysis and statistics.
  1966.

\bibitem{LMN15}
J.~Llibre, A.~C. Mereu, and D.~D. Novaes.
\newblock Averaging theory for discontinuous piecewise differential systems.
\newblock {\em J. Differential Equations}, 258(11):4007--4032, 2015.

\bibitem{JN3}
J.~Llibre and D.~D. Novaes.
\newblock Improving the averaging theory for computing periodic solutions of
  the differential equations.
\newblock {\em Z. Angew. Math. Phys.}, 66(4):1401--1412, 2015.

\bibitem{LliNovRod2017}
J.~Llibre, D.~D. Novaes, and C.~A.~B. Rodrigues.
\newblock Averaging theory at any order for computing limit cycles of
  discontinuous piecewise differential systems with many zones.
\newblock {\em Phys. D}, 353/354:1--10, 2017.

\bibitem{LliNovRod2020}
J.~Llibre, D.~D. Novaes, and C.~A.~B. Rodrigues.
\newblock Bifurcations from families of periodic solutions in piecewise
  differential systems.
\newblock {\em Phys. D}, 404:132342, 12, 2020.

\bibitem{LNT2014}
J.~Llibre, D.~D. Novaes, and M.~A. Teixeira.
\newblock Higher order averaging theory for finding periodic solutions via
  {B}rouwer degree.
\newblock {\em Nonlinearity}, 27(3):563--583, 2014.

\bibitem{LRT}
J.~Llibre, S.~Rebollo-Perdomo, and J.~Torregrosa.
\newblock Limit cycles bifurcating from isochronous surfaces of revolution in
  {$\Bbb R^3$}.
\newblock {\em J. Math. Anal. Appl.}, 381(1):414--426, 2011.

\bibitem{Ma}
I.~G. Malkin.
\newblock {\em Some problems of the theory of nonlinear oscillations}.
\newblock Gosudarstv. Izdat. Tehn.-Teor. Lit., Moscow, 1956.

\bibitem{RC}
M.~B.~H. Rhouma and C.~Chicone.
\newblock On the continuation of periodic orbits.
\newblock {\em Methods Appl. Anal.}, 7(1):85--104, 2000.

\bibitem{Ro}
M.~Roseau.
\newblock {\em Vibrations non lin\'eaires et th\'eorie de la stabilit\'e},
  volume Vol. 8 of {\em Springer Tracts in Natural Philosophy}.
\newblock Springer-Verlag, Berlin-New York, 1966.

\bibitem{THM}
H.~Tian and M.~Han.
\newblock Bifurcation of periodic orbits by perturbing high-dimensional
  piecewise smooth integrable systems.
\newblock {\em Journal of Differential Equations}, 263(11):7448 – 7474, 2017.
\newblock Cited by: 75; All Open Access, Bronze Open Access.

\end{thebibliography}

\end{document}